\documentclass[11pt,twoside]{amsart}
\usepackage[backref]{hyperref}
\usepackage{amssymb,amsthm,amsmath,amstext}
\usepackage{mathrsfs}  
\usepackage{bm}        
\usepackage{mathtools} 
\usepackage{graphicx}
\usepackage[all]{xy}
\usepackage{color}

\usepackage[left=1.5in,top=1.5in,right=1.5in,bottom=1.55in]{geometry}

\usepackage{amscd}
\usepackage[latin1]{inputenc}
\usepackage[OT1]{fontenc}
\usepackage[bbgreekl]{mathbbol}

\theoremstyle{plain}
\newtheorem{theorem}{Theorem}[section]

\newtheorem{proposition}[theorem]{Proposition}
\newtheorem{lemma}[theorem]{Lemma}

\theoremstyle{definition}
\newtheorem{definition}[theorem]{Definition}

\newtheorem{remark}[theorem]{Remark}

\newcommand{\sheaf}[1]{\mathscr{#1}}

\newcommand{\PP}{\sheaf{P}}

\newcommand{\CC}{\sheaf{C}}
\newcommand{\GG}{\sheaf{G}}

\newcommand{\Z}{\mathbb Z}

\newcommand{\C}{\mathbb C}
\renewcommand{\P}{\mathbb P}

\newcommand{\F}{\mathbb{F}}

\DeclareMathOperator{\rk}{\mathrm{rk}}

\newcommand{\StabH}{H_{\mathrm{s}}}

\newcommand{\et}{\mathrm{\acute{e}t}}

\usepackage{hyperref}

\begin{document}

\title[Stable cohomology of central extensions]{On stable cohomology of central extensions
of elementary abelian groups}





\author[Bogomolov]{Fedor Bogomolov}
\address{Fedor Bogomolov, Courant Institute of Mathematical Sciences,
New York University\\ 
251 Mercer Street\\ 
New York, NY 10012, USA}
\email{bogomolo@cims.nyu.edu}

\author[B\"ohning]{Christian B\"ohning}
\address{Christian B\"ohning, Mathematics Institute, University of Warwick\\
Coventry CV4 7AL, England}
\email{C.Boehning@warwick.ac.uk}

\author[Pirutka]{Alena Pirutka}
\address{Alena Pirutka, Courant Institute of Mathematical Sciences,
New York University\\ 
251 Mercer Street\\ 
New York, NY 10012, USA}
\email{pirutka@cims.nyu.edu}

\begin{abstract}
We study when kernels of inflation maps associated to extraspecial $p$-groups in stable group cohomology are generated by their degree two components. This turns out to be true if the prime is large enough compared to the rank of the elementary abelian quotient, but false in general. 
\end{abstract}

\maketitle

\section{Introduction and statement of results}\label{sIntro}

Throughout below $k$ will be an algebraically closed field of characteristic $l\ge 0$ and $p$ will be a prime number assumed to be different from $l$ if $l$ is positive. Let $G$ be a finite $p$-group. One defines the stable cohomology $H^*_{\mathrm{s}, k} (G, \Z/p) = H^*_{\mathrm{s}}(G, \Z/p)$ in the following way (this does depend on $k$, but we suppress it from the notation when there is no risk of confusion): for a finite-dimensional generically free linear $G$-representation $V$, let $V^L\subset V$ be the open subset where $G$ acts freely. Then the ideal $I_{G, \mathrm{unstable}}$ in the group cohomology ring $H^* (G, \Z/p)$ is defined to be, equivalently, the kernel of the natural homomorphism
\begin{gather}\label{fUnstable1}
H^* (G, \Z/p) \to H^* (\mathrm{Gal}(k (V/G)), \Z/p)
\end{gather}
or, more geometrically, the kernel of 
\begin{gather}\label{fUnstable2}
H^* (G, \Z/p) \to \varinjlim\limits_{U\subset V^L/G} H^*_{\et} (U, \Z/p)
\end{gather}
where $U$ runs over all nonempty Zariski open subsets of $V^L/G$. 

\begin{definition}\label{dStableCohomology}
We define $H^*_{\mathrm{s}}(G, \Z/p)$ as 
\[
H^*_{\mathrm{s}}(G, \Z/p) = H^* (G, \Z/p)/I_{G, \mathrm{unstable}}.
\]
\end{definition}

A priori, this seems to depend on the choice of $V$, but really does not \cite[Thm. 6.8]{Bo05}. We often identify $H^*_{\mathrm{s}}(G, \Z/p)$ with 
its image in $H^*(\mathrm{Gal}(k(V/G), \Z/p)$.

\begin{definition}\label{dUnramified}
Put $L=k(V/G)$. The unramified group cohomology 
\[
H^*_{\mathrm{nr}} (G, \Z/p) \subset H^*_{\mathrm{s}}(G, \Z/p)
\]
is defined as the intersection, inside $H^* (L, \Z/p)$, of $H^*_{\mathrm{s}}(G, \Z/p)$ and $H^*_{\mathrm{nr}}(L, \Z/p)$; here, as usual, $H^*_{\mathrm{nr}}(L, \Z/p)$ are those classes that are in the kernel of all residue maps associated to divisorial valuations of $L$, i.e. those corresponding to a prime divisor on some normal model of $L$. 
\end{definition}

In this article we study a rather special class of groups. 

\begin{definition}\label{dExtraspecial}
For a prime $p$, an extraspecial $p$-group $G$ is a $p$-group such that its center $Z(G)$ is cyclic of order $p$ and $G/Z(G)$ is a nontrivial elementary abelian group. 
\end{definition}

This differs a bit from the arguably most common definition using the Frattini subgroup \cite[4., \S 4, Def. 4.14]{Suz86}, but it is equivalent to it by \cite[4., \S 4, 4.16]{Suz86}. 

Thus each extraspecial $p$-group sits in an exact sequence
\begin{gather}\label{fExtensionExtraspecial}
\xymatrix{
1 \ar[r] & Z\ar[r] & G \ar[r]^{\pi} &  E \ar[r] & 1
}
\end{gather}
where $Z\simeq \Z /p$ is the center of the group $G$ and $E \simeq (\Z/p)^n$ is elementary abelian. Moreover, the skew-form given by taking the commutator of lifts of elements in $E$ 
\begin{align}\label{fCommutatorSkewForm}
\omega \colon & E \times E \to Z, \\ \nonumber
(x,y) & \mapsto [\tilde{x}, \tilde{y}]
\end{align}
must be a symplectic form if $G$ is extraspecial. Hence $n=2m$ and the order of $G$ is of the form $p^{1+2m}$ for some positive integer $m$. One can be much more precise and prove that for each given order $p^{1+2m}$ there are precisely two extraspecial $p$-groups of that given order, up to isomorphism \cite[III, \S 13, 13.7 and 13.8]{Hupp67} or \cite[Chapter 5, 5.]{Gor07}; but we do not need this detailed structure theory. We want to study the kernel of the ``inflation map"
\begin{gather}
K^G = \mathrm{ker} \left( \pi^* \colon H^*_{\mathrm{s}} (E, \Z/p ) \to H^*_{\mathrm{s}} (G, \Z/p) \right) .
\end{gather}
This is a graded ideal in the graded ring $H^*_{\mathrm{s}} (E, \Z/p )$ (graded by cohomological degree). 
It is natural to expect that this should be, in general, generated by its degree $2$ component, or, even more precisely, by the class $\omega\in \mathrm{Hom}(\Lambda^2 E, \Z/p)=H^2_{\mathrm{s}}(E, \Z/p)$ given by the extension; cf. also formula (\ref{fStableCohAbelian}) in Section (\ref{sProofs}) for the description of the stable cohomology of abelian groups. In fact, Tezuka and Yagita in \cite{TezYag11} study a very similar problem in \S 9, p. 4492 ff., see especially the problems they mention on p. 4494 top and bottom concerning what they cannot yet prove. Indeed, the expectation above is false in general (this is similar to the situation in ordinary group cohomology where conjectures that kernels of inflation maps associated to central extensions should always be the expected ones are false as well, see \cite[Prop. 9]{Rus92}). We show:

\begin{theorem}\label{tMain1}
Let $G$ be an extraspecial $p$-group of order $p^{1+2m}$ as above. Then, provided $p > m$, the ideal $K^G$ is generated by $K_2^G= \langle \omega\rangle $. 
\end{theorem}

On the other hand:

\begin{theorem}\label{tMain2}
Take $k=\C$. 
If $G_0$ is the extraspecial $2$-group of order $2^{1+6}$ that is the preimage, under the natural covering map
\[
\mathrm{Spin}_7\, (k) \to \mathrm{SO}_7 (k)
\]
of the diagonal matrices $\mathrm{diag} (\pm 1, \dots , \pm 1)$ in $\mathrm{SO}_7 (k)$, then $K^{G_0}$ is \emph{not} generated by its degree $2$ piece $K^{G_0}_2$; here again, $K ^{G_0}_2=\langle \omega \rangle$. 
\end{theorem}

This does not seem to be related to the fact that $p=2$ is a special prime; we believe similar examples could very likely be given for every other prime $p$ as well, as will become apparent from the construction in the proof.

\begin{remark}\label{rRelevance}
Theorems \ref{tMain1} and \ref{tMain2} should be seen in the following context, which provided us with motivation for this work.
\begin{enumerate}
\item
As pointed out in \cite[Thm. 11]{BT11}, the Bloch-Kato conjecture (Voevodsky's theorem) implies that, letting $\Gamma=\mathrm{Gal} (k(V/G))$ as above, and denoting 
\[
\Gamma^a = \Gamma/[\Gamma, \Gamma], \quad \Gamma^c= \Gamma/[\Gamma, [\Gamma, \Gamma]],
\]
the natural map $H^* (\Gamma^a, \Z/p) \to H^*(\Gamma , \Z/p)$ is surjective, its kernel $K^{\Gamma^a}$ coincides with the kernel of $H^* (\Gamma^a, \Z/p) \to H^*(\Gamma^c , \Z/p)$, and 
is generated by its degree two component $K^{\Gamma^a}_2$; this follows not obviously from a spectral sequence argument, but in any case directly from the Bloch-Kato conjecture since for $L=k(V/G)$
\[
H^n (\Gamma^a, \Z/p) \simeq (L^*\otimes_{\Z} \dots \otimes_{\Z} L^*)/p, \quad H^n (\Gamma , \Z/p) \simeq K_n (L)/p 
\]
and the Milnor K-group $K_n(L)$ is a quotient of $L^*\otimes_{\Z} \dots \otimes_{\Z} L^*$ by the $n$-th graded piece of the ideal generated by the Steinberg relations in degree two. Thus, whereas on the full profinite level, kernels of inflation maps are generated in degree two, this property is not inherited by finite quotients of the full Galois group.
\item
The importance to understand central extensions of abelian groups for the computation of stable and unramified cohomology in general is explained, for example, in \cite[\S 8]{BT11}. 
The inclusion mentioned in \cite{BT17} after formula 1.2 can be strict.
\end{enumerate}
\end{remark}

\medskip

\noindent
\textbf{Acknowledgments.} The first author was partially supported by the Russian Academic
Excellence Project `5-100', by a Simons Fellowship  and
by  EPSRC programme grant EP/M024830. The third author was partially
supported by NSF grant DMS-1601680, by 
ANR grant ANR-15-CE40-0002-01, and by the Laboratory of Mirror Symmetry NRU HSE, RF Government grant, ag.$\backslash$no.$\backslash$14.641.31.0001.

\section{Some linear algebra}\label{sLinAlg}
We establish some results  concerning the exterior algebra of a symplectic vector space over a field of any characteristic. Most of this is contained in \cite[Ch. VIII, \S 13, 3., p. 203-210]{BourLie}, but since the standing assumption in loc. cit., Ch. VIII, is to work over a field of characteristic $0$ whereas we are interested in the case of a base field of finite characteristic, it is necessary to point out in detail which statements go through unchanged and which ones require adaptation. 

Let $\F$ be any field, and let $V$ be a finite-dimensional $\F$-vector space of even dimension $n=2m$. Suppose that $V$ is symplectic, which means endowed with a non-degenerate alternating bilinear form $\Psi$. Let $\mathrm{Sp}_{2m}(V, \Psi )= \mathrm{Sp}_{2m}$ be the corresponding symplectic group. From e.g. \cite[Prop. 1.8]{EKM08} it follows that $V$ is isometric to an orthogonal direct sum of $m$ hyperbolic planes, in other words there exists a symplectic basis 
\[
(e_1, \dots , e_m, e_{-m}, \dots, e_{-1})
\]
with $\Psi (e_i, e_j) = 0$ unless $i=-j$ when $\Psi (e_i, e_{-i})=1$. This is a statement entirely independent of the characteristic of $\F$, in particular, also holds in characteristic two (the form is then at the same time alternating and symmetric). Let $V^*$ be the dual vector space to $V$, and $(e_i^*)$ the basis dual to the basis $(e_i)$. We identify the alternating form $\Psi$ with an element $\Gamma^*\in \Lambda^2 V^*$. Then 
\begin{gather}\label{fGammaStar}
\Gamma^* = - \sum_{i=1}^m e_i^*\wedge e^*_{-i}. 
\end{gather}
Via the isomorphism $V \to V^*$ given by $\Psi$, the form $\Psi$ induces a symplectic form $\Psi^*$ on $V^*$. Identifying $\Psi^*$ with an element $\Gamma$ in $\Lambda^2 V$, one finds
\begin{gather}\label{fGamma}
\Gamma = \sum_{i=1}^m e_i\wedge e_{-i}. 
\end{gather}
One also denotes by $X_-\colon \Lambda^* V \to \Lambda^* V$ the endomorphism induced by left exterior product with $\Gamma$ and by $X_{+}\colon  \Lambda^* V \to \Lambda^* V$ the endomorphism given by left interior product (contraction) with $-\Gamma^*$; more precisely,
\begin{gather}\label{fContraction}
X_+ (v_1\wedge \dots \wedge v_r) = \sum_{1\le i < j \le r} (-\Psi)(v_i, v_j) (-1)^{i+j} v_1\wedge \dots \wedge \hat{v}_i \wedge \dots \wedge \hat{v}_j\wedge \dots \wedge v_r .
\end{gather}
Moreover, let $H \colon \Lambda^* V \to \Lambda^* V$ be the endomorphism that is multiplication by $(m-r)$ on $\Lambda^r V$ for $0\le r \le 2m$. Then as in \cite[p. 207 and Ex. 19]{BourLie}, it follows that
\begin{align}\label{fSL2}
[X_+, X_-]= -H, \quad [H, X_+]=2X_+, \quad [H, X_-]=-2X_-
\end{align}
so the vector subspace of $\mathrm{End}(\Lambda^* V)$ generated by $X_+, X_-, H$ is a Lie subalgebra isomorphic to $\mathfrak{sl} (2, \F)$. Moreover, for the action of $\mathfrak{sl} (2, \F)$ on $\Lambda^* V$, the subspace $\Lambda^r V$ is the subspace of elements of weight $m-r$.

\medskip

In the following Proposition and its proof, we make the conventions that for integers $i < 0$, $\Lambda^i V := 0$ and for binomial coefficients and positive integers $n$, ${n \choose i }:=0$.

\begin{proposition}\label{pDecompositionUnderSL2}
Put $E_r = (\Lambda^r V) \cap \mathrm{ker}\, X_+$, the ``primitive elements" in $\Lambda^r V$. 
If $p=\mathrm{char}\, \F > \dim V/2 = m$ or $\mathrm{char}\, \F =0$, then 
\begin{enumerate}
\item
for $r \le m-1$, the restriction of $X_-$ to $\Lambda^r V$ is injective;
\item
for $r \ge m-1$, the restriction of $X_-$ to $\Lambda^r V$ induces a surjection from $\Lambda^r V$ onto $\Lambda^{r+2} V$;
\item
for $r\le m$, 
\[
\Lambda^r V = E_r \oplus X_-(\Lambda^{r-2} V) .
\]
\end{enumerate}
Moreover, $E_r$ coincides with the submodule $F_r \subset \Lambda^r V$ defined as the span of all ``completely reducible" $r$-vectors $v_1\wedge \dots \wedge v_r$ such that $\langle v_1, \dots , v_r \rangle$ is a totally isotropic subspace of $V$. Here completely reducible means simply a pure wedge product of the above form $v_1\wedge \dots \wedge v_r$.
\end{proposition}

\begin{proof}
The proof is based on the following observations.
\begin{enumerate}
\item[(I)]
Let $E$ be any $\mathfrak{sl} (2, \F)$-module, and let $\epsilon$ be a primitive element, by which we mean, as usual, $X_+(\epsilon )=0$ and $\epsilon$ is an eigenvector for some $\lambda\in \F$ for $H$. Then, as long as $\nu$ is an integer such that $1\le \nu < p$ it does make sense to define
\[
\epsilon_{\nu}= \frac{(-1)^\nu}{\nu}X_-^{\nu} \epsilon, \quad \epsilon_0=\epsilon, \quad  \epsilon_{-1}=0. 
\]
Then a straightforward computation with the commutation relations (\ref{fSL2}), done in \cite[Chapter VIII, \S 1, 2., Prop. 1]{BourLie}, shows that 
\begin{align}
H\epsilon_{\nu} &=(\lambda - 2\bar{\nu}) \epsilon_{\nu} \nonumber \\
X_-\epsilon_{\nu} &= -(\bar{\nu}+1) \epsilon_{\nu +1} \label{fLadderOperators} \\ 
X_+\epsilon_{\nu} &= (\lambda - \bar{\nu} +1) \epsilon_{\nu -1} \nonumber
\end{align}
as long as all indices of the occurring $\epsilon$'s are $< p$. Here we put a bar on an integer to indicate that we consider it as an element of $\F$ via the natural homomorphism $\Z \to \F$, which for us will however be usually not injective.
\item[(II)]
If we define $F_r$ as in the statement of the Proposition, then obviously $F_r\subset E_r$ and
\begin{gather}\label{fDimFr}
\dim F_r= {2m \choose r} - {2m \choose r-2} , \quad 0\le r \le m.
\end{gather}
This is proven in \cite[Thm. 1.1]{Bruyn09} under no assumptions on $p=\mathrm{char}\, \F$. 
\end{enumerate}
For the module $E=\Lambda^* V$ we can thus display the action of the operators $H, X_+, X_-$ schematically in the familiar way:

\[
\xymatrix{
 & \Lambda^{0} V \ar@{->}@(dl,dr)_H \ar@/^/[r]^{X_-}& \Lambda^{2} V \ar@{->}@(dl,dr)_H \ar@/^/[l]^{X_+}  & \dots  & \Lambda^{2m-2} V\ar@{->}@(dl,dr)_H \ar@/^/[r]^{X_-} & \Lambda^{2m} V\ar@{->}@(dl,dr)_H \ar@/^/[l]^{X_+} \\
\textbf{weight:} & \bf{\overline{m}} & \bf{\overline{m-2}} & \dots & \bf{\overline{-(m-2)}} & \bf{\overline{- m}}
}
\]

and

\[
\xymatrix{
 &  V \ar@{->}@(dl,dr)_H \ar@/^/[r]^{X_-}& \Lambda^{3} V \ar@{->}@(dl,dr)_H \ar@/^/[l]^{X_+} & \dots & \Lambda^{2m-3} V\ar@{->}@(dl,dr)_H \ar@/^/[r]^{X_-} & \Lambda^{2m-1} V\ar@{->}@(dl,dr)_H \ar@/^/[l]^{X_+} \\
\textbf{weight:} & \bf{\overline{m-1}} & \bf{\overline{m-3}} & \dots & \bf{\overline{-(m-3)}} & \bf{\overline{- (m-1)}}
}
\]

\medskip

Now we start to use the assumption that $p=\mathrm{char}\, \F > m$. 

If we start with a primitive element $\epsilon=\epsilon_0$ in one of the $F_r$, $0\le r \le m$, of weight $\lambda = \overline{m-r}$ in $\{ \bar{0}, \dots , \bar{m} \}$, then the $\epsilon_{\nu}$ as in item (I) above are all defined for $\nu =0, \dots , m$. Moreover, if $\mu$ is the largest integer such that $\epsilon_{\mu}\neq 0$, then $\mu \le m < p$ and $\mu$ can only possibly be equal to $m$ if we start with $\epsilon_0$ in $F_0$; excluding the latter case for a moment, we can use the third of formulas (\ref{fLadderOperators}) to get
\[
0 = X_+(\epsilon_{\mu +1} ) = (\lambda - \bar{\mu}) \epsilon_{\mu}
\]
where now all indices are still $<p$, and one can only have that $\lambda - \bar{\mu}=0$ in $\F$ if $\mu$ is the unique lift of $\lambda$ in $\{0, \dots , m\}$. If $\mu=m$ and $\epsilon_0\in F_0$, the third of formulas (\ref{fLadderOperators}) still shows that all of $\epsilon_0, \epsilon_1, \dots , \epsilon_m$ must be nonzero (since applying $X_+$ to any of them the appropriate number of times returns a nonzero multiple of $\epsilon_0$ under the standing assumptions), so putting all this together we can say that starting from a primitive $\epsilon_0\in F_r$, $0\le r \le m$, with weight $\lambda=\overline{m-r}$ we get a chain of nonzero vectors

\[
\xymatrix{
 &  \epsilon_0  \ar@/^/[r]^{X_-}& \epsilon_1  \ar@/^/[l]^{X_+} & \dots & \epsilon_{m-r-1} \ar@/^/[r]^{X_-} & \epsilon_{m-r} \ar@/^/[l]^{X_+} \\
weight: & \overline{m-r} & \overline{m-r-2} & \dots & \overline{-(m-r-2)} & \overline{- (m-r)}
}
\]
where according to the formulas in (\ref{fLadderOperators}) and since $p> m$, the $X_+$ and $X_-$ map each of the $\epsilon$'s to a nonzero multiple of the subsequent one ``up or down the ladder" as indicated in the previous diagram. In particular, $X_+$ and $X_-$ induce isomorphisms between the vector subspaces indicated in the following diagram:

\[
\xymatrix{
 &  F_r  \ar@/^/[r]^{X_-}& X_-(F_r)  \ar@/^/[l]^{X_+} & \dots & X_-^{m-r-1}(F_r) \ar@/^/[r]^{X_-} & X_-^{m-r} (F_r) \ar@/^/[l]^{X_+}
 }
\]

for $0\le r\le m$. In addition, the sum of $F_r$, $X_-(F_{r-2})$, $X^2_-(F_{r-4}), \dots$ inside $\Lambda^r V$ (for any $0\le r \le 2m$, noting $F_s=0$ for $s> m$) is direct: this can be seen by repeatedly applying $X_+$ and using the third formula of (\ref{fLadderOperators}). Then a dimension count using item (II) at the beginning of the proof yields
\[
\Lambda^r V = F_r \oplus X_-(F_{r-2}) \oplus X^2_-(F_{r-4}) \oplus \dots = F_r \oplus X_-(\Lambda^{r-2}V)
\]
for any $r$, which proves $E_r=F_r$, c) in the statement of the Proposition as well as a) and b).
\end{proof}

We will need one further piece of information concerning $E_r=F_r$, $0\le r \le m$ later.

\begin{theorem}\label{tPremetSuprunenko}
The $\mathrm{Sp}(2m, \F)$-module $E_r=F_r$, $1\le r \le m$ is irreducible if 
\[
p > m - \frac{r}{2} +1
\]
and if and only if $p$ does not divide
\[
\prod_{0\le j \le r, j \equiv r \, (\mathrm{mod}\, 2)} {m- \frac{r+j}{2} +1 \choose (r-j)/2} .
\]
\end{theorem}

\begin{proof}
This is \cite[p. 1313, Thm. 2]{PS83}.
\end{proof}

\section{Proofs of main results}\label{sProofs}

We start by recalling that for an abelian $p$-group $A$ 

\begin{gather}\label{fStableCohAbelian}
H^*_{\mathrm{s}}(A, \Z /p) \simeq \Lambda^* A^{\vee} = \Lambda^* H^1 (A, \Z /p) = \Lambda^* \mathrm{Hom}(A, \Z/p).
\end{gather}

See e.g. \cite{Bo05}, Example after Remark 6.10. We now turn to extraspecial groups $G$ sitting in an exact sequence (\ref{fExtensionExtraspecial}), and retain the notation from Section \ref{sIntro}. 

\begin{definition}\label{dTotallyIsotropicSubgroup}
A subgroup $A\subset E$ is called \emph{totally isotropic} if it is a totally isotropic subspace of the $\F_p$-vector space $E$, i.e. the symplectic form $\omega$ vanishes identically on $A$.
\end{definition}

Totally isotropic subgroups $A$ can be characterised as the ones such that $\tilde{A}=\pi^{-1}(A)\subset G$ is abelian of the same rank, where $\pi$ is the natural surjection $\pi\colon G \to E$. 

\begin{proof}[Proof of Theorem \ref{tMain1}]
For a totally isotropic subgroup $A$ of $E$, consider the diagram
\begin{gather}\label{fDiagramRestrictionIsotropic}
\xymatrix{
H^*_{\mathrm{s}}(E, \Z /p)\ar[d]^{r^{E}_A} \ar[r]^{\pi^*} & H^*_{\mathrm{s}}(G, \Z/p)\ar[d]^{r^G_{\tilde{A}}} \\
H^*_{\mathrm{s}} (A, \Z /p) \ar[r]^{\pi^*} & H^*_{\mathrm{s}} (\tilde{A}, \Z/p)
}
\end{gather}
where the vertical arrows are the restriction maps. From the description of the stable cohomology of abelian groups, one gets that the lower horizontal arrow is injective, and in fact an isomorphism. In other words, a class $\alpha \in H^r_{\mathrm{s}}(E, \Z /p)$ that is nontrivial on a totally isotropic subgroup is \emph{not} in the kernel of $\pi^*$. Applying Proposition \ref{pDecompositionUnderSL2} to the symplectic vector space $V= H^1 (E, \Z/p) = E^*$, we see that in order to prove the Theorem it suffices to show that every nonzero class $\alpha\in E_r$, $0\le r\le m$, is nontrivial on some totally isotropic subgroup. Since totally isotropic subgroups are invariant under the action of $\mathrm{Sp}_{2m}(\F_p )$, the classes in $E_r$ that are trivial on all totally isotropic subgroups form a $\mathrm{Sp}_{2m}(\F_p )$ submodule; as $E_r$ is irreducible by Theorem \ref{tPremetSuprunenko}, this submodule is either reduced to zero or everything. Hence it suffices to prove that some class $\alpha\in E_r$, for every $0\le r\le m$, is nontrivial on some totally isotropic subgroup. But this is clear: in the notation introduced at the beginning of Section \ref{sLinAlg}, if we take the totally isotropic subgroup $A=\langle e_1^*, \dots , e_r^*\rangle$, then $\alpha = e_1\wedge \dots \wedge e_r$ is nontrivial on it. 
\end{proof}

For the proof of Theorem \ref{tMain2} we need a few more auxiliary results. Assume $k=\C$ now. As in the statement of that Theorem, consider the preimage $G_0\subset \mathrm{Spin}_7 (k)$ of the diagonal matrices with entries $\pm 1$ in $\mathrm{SO}_7 (k)$.

\begin{lemma}\label{lG0Extraspecial}
$G_0$ is an extraspecial $2$-group sitting in an exact sequence
\[
0 \to \Z/2 \to G_0 \to (\Z/2)^6 \to 0.
\]
\end{lemma}

\begin{proof}
The existence of such an exact sequence is clear. The point is to verify that $\Z/2$ is the entire center of $G_0$, and this is done in \cite[Chapter 4, Lemma on p. 22]{Ad96}; for this it is important that $7$ coming from $\mathrm{Spin}_7 (k)$ is odd: the center of the analogously defined groups for the even Spin-groups is larger (the claim that all of them are extraspecial in \cite[5.5, p. 154]{Bak02} is erroneous).
\end{proof}

\begin{lemma}\label{lStableEquivalenceG0Spin7}
Generically free linear $\mathrm{Spin}_7 (k)$-quotients $V/\mathrm{Spin}_7(k)$ and generically free linear $G_0$-quotients $W/G_0$ are stably equivalent. Here $V$ resp. $W$ are any generically free (finite-dimensional, complex) linear representations of $\mathrm{Spin}_7 (k)$ resp. $G_0$.
\end{lemma}

\begin{proof}
This is proven in \cite[\S 3 ff.]{Bo87}, but we include the easy argument for the sake of completeness and give a few more details. 
Consider the standard representation $k^7$ of $\mathrm{SO}_7(k)$; via the natural covering map it is a $\mathrm{Spin}_7 (k)$-representation. Inside $k^7\oplus\dots \oplus k^7$ (seven times), consider the subvariety $R$ of tuples of vectors $(v_1, \dots , v_7)$ that are mutually orthogonal. This is invariant under the group action, and has the structure of a tower of equivariant vector bundles over any of the summands $k^7$. Let $P=ke_1\times \dots \times k e_7\subset  R$ be the Cartesian product of the lines through the standard basis vectors $e_1, \dots , e_7$. Then $P$ is a $(\mathrm{Spin}_7(k), G_0)$-section of the action and \cite[Theorem 3.1]{CTS07} applies; in particular, given any generically free linear $\mathrm{Spin}_7 (k)$-representation $V$, then (a) $(V\oplus R)/\mathrm{Spin}_7(k)$ is stably equivalent to $V/\mathrm{Spin}_7 (k)$ since $R$ has the structure of a tower of equivariant vector bundles and one can apply the ``no-name lemma" \cite[Thm. 3.6]{CTS07}; (b) in $V\oplus R$ the subvariety $V\times P$ is a $(\mathrm{Spin}_7(k), G_0)$-section, whence $(V\oplus R)/\mathrm{Spin}_7(k)$ is birational to $(V\times P )/G_0$. This concludes the proof. 
\end{proof}

\begin{theorem}\label{tTrivUnramified}
Generically free linear $\mathrm{Spin}_7 (k)$-quotients are stably rational; in particular, combining this with Lemma \ref{lStableEquivalenceG0Spin7}, $G_0$ has trivial unramified cohomology.
\end{theorem}

\begin{proof}
The fact that generically free $\mathrm{Spin}_7\, \C$-quotients are stably rational is proven in \cite{Kor00} (despite the title only referring to $\mathrm{Spin}_{10}$), see also \cite[\S 4.5]{CTS07}. The fact that unramified cohomology of stably rational varieties is trivial is proven, for example, in \cite[Prop. 4.1.4]{CT95}.
\end{proof}

\begin{proposition}\label{pCriterionUnramified}
Let $G$ be a finite group. Suppose that $\alpha\in H^*_{\mathrm{s}}(G, \Z/p)$ is a class whose restriction, for any $g\in G$, to $H^*_{\mathrm{s}}(Z(g), \Z/p)$ is induced from $H^*_{\mathrm{s}}(Z(g)/\langle g\rangle, \Z/p)$; here $Z(g)$ is the centraliser of $g$ in $G$. Then $\alpha$ is unramified. 
\end{proposition}

\begin{proof}
This follows from the way residue maps in Galois cohomology are defined, see \cite[Chapter II, 7.]{GMS} for the following: for $K=k(V)^G$, $V$ a generically free $G$-representation, and $v$ a geometric discrete valuation of $K$, one considers the completion $K_v$, and the decomposition group $\mathrm{Dec}_w$ where $w$ is an extension of $v$ to the separable closure $K_s$. Then $\mathrm{Gal}(K_v) \simeq \mathrm{Dec}_w \subset \mathrm{Gal}(K)$, and there is a split exact sequence 
\begin{gather}\label{fExtensionDecomposition}
1\to I \to \mathrm{Gal}(K_v) \to \mathrm{Gal}(\kappa_v)\to 1
\end{gather}
where $\kappa_v$ is the residue field of $v$ and $I \simeq \hat{\Z}$ is the topologically cyclic inertia subgroup. For a finite constant $\mathrm{Gal}(K)$-module $C$ of order not divisible by $\mathrm{char}(k)$, there is an exact sequence 
\[
\xymatrix{
0 \ar[r] & H^i (\mathrm{Gal}(\kappa_v), C)\ar[r] &  H^i (\mathrm{Gal}(K_v), C) \ar[r]^{r\quad\quad} & H^{i-1}(\mathrm{Gal}(\kappa_v), \mathrm{Hom}(I, C)) \ar[r] & 0
}
\]
and $r$ is the residue map. The residue of an element $\beta \in H^i (K, C)$ is obtained by restricting to $H^i (K_v, C)$ and afterwards applying $r$. Now under the natural map 
\[
\mathrm{Gal}(K) \to G
\]
the topologically cyclic inertia subgroup $I$ will map to a cyclic subgroup of $G$ generated by some element $g\in G$, and (\ref{fExtensionDecomposition}) being a central extension, $\mathrm{Gal}(K_v)$ will map into the centraliser $Z(g) \subset G$. Now if $C=\Z/p$ and $\alpha \in H^*_{\mathrm{s}}(G, \Z/p)$ is a class whose restriction to $H^*_{\mathrm{s}}(Z(g), \Z/p)$ is induced from $H^*_{\mathrm{s}}(Z(g)/\langle g\rangle, \Z/p)$, then since there is a commutative diagram,
\[
\xymatrix{
\mathrm{Gal}(\kappa_v)\ar[d] & \mathrm{Gal}(K_v)\ar[d]^{\psi} \ar@{->>}[l] \ar@{^{(}->}[r] & \mathrm{Gal}(K) \ar@{->>}[d]^{\tau} \\
    Z(g)/\langle g\rangle &       Z(g)\ar@{->>}[l] \ar@{^{(}->}[r]        & G
}
\]
and a factorisation
\[
\xymatrix{
H^i (Z(g), \Z/p) \ar[rr] \ar[rd] & & H^i (\mathrm{Gal}(K_v), \Z/p)\\
 & H^i_{\mathrm{s}} (Z(g), \Z/p) \ar[ru] &
}
\]
(the latter because $\mathrm{Gal}(K_v)$ sits inside $\mathrm{Gal}(k(V/(Z(g))) \subset \mathrm{Gal}(K)$, $\mathrm{Gal}(k(V/(Z(g)))$ being the preimage of $Z(g)$ under $\tau$), we get that the restriction of $\alpha$ to the decomposition group comes from $H^i (\mathrm{Gal}(\kappa_v), \Z/p)$, hence its residue is zero.
\end{proof}

\begin{proof}[Proof of Theorem \ref{tMain2}]
Let $E=(\Z/2)^6$, $V=E^*$. On $E$ we choose coordinates $x_1, x_2, x_3, y_1, y_2, y_3 \in E^*$ which form a symplectic basis and so that
\[
\omega= \sum_{i=1}^3 x_i\wedge y_i .
\]
To prove the Theorem we are going to proceed in the following Steps. 
\begin{enumerate}
\item[\textbf{Step 1.}]
We produce a class $\zeta \in H^4_{\mathrm{s}}(E, \Z/2)=\Lambda^4 V$ that is not in the ideal generated by $\omega$. More precisely, we will take
\[
\zeta= x_2\wedge x_3\wedge y_2\wedge y_3. 
\]
\item[\textbf{Step 2.}]
We prove that $\pi^* (\zeta ) \in H^4_{\mathrm{s}}(G_0, \Z/2)$ is unramified using the criterion given in Proposition \ref{pCriterionUnramified}. By Theorem \ref{tTrivUnramified}, we conclude $\pi^* (\zeta) =0$ whence $\zeta$ is a class in the kernel of $\pi^*$ not in the ideal generated by $\omega$.
\item[\textbf{Step 3.}]
We check that the kernel of $\pi^* \colon H^2_{\mathrm{s}}(E, \Z/2) \to H^2_{\mathrm{s}}(G_0, \Z/2)$ is spanned by $\omega$.
\end{enumerate}
The assertion in Step 1 is proved by direct computation, which can be done either by hand, or, as we did, using Macaulay2. For example, the following code does that:
\tiny
\begin{verbatim}
i1 : R=ZZ/2[x1,x2,x3,y1,y2,y3,SkewCommutative=>true,Degrees=>{1,1,1,1,1,1}]

o1 = R

o1 : PolynomialRing

i2 : M=matrix{{x1*y1+x2*y2+x3*y3}}

o2 = | x1y1+x2y2+x3y3 |

             1       1
o2 : Matrix R  <--- R

i3 : basis(4,cokernel M)

o3 = | x2x3y2y3 |

o3 : Matrix
\end{verbatim}
\normalsize

\medskip

For Step 2, one uses the following observation, which can be proved by direct computation, either by hand or using Macaulay2 again: 

\begin{lemma}\label{lMinimal}
Suppose $V$ is a symplectic vector space over $\F_2$ of dimension $n=2m \le 4$. Then, for $r\ge m-1$, the restriction of $X_-$ to $\Lambda^{r}V$ is surjective.
\end{lemma}

In other words, an extraspecial $2$-group of minimal order for which Proposition \ref{pDecompositionUnderSL2}, b), fails has order that of $G_0$. We now use this to prove the assertion in Step 2. 

\medskip

For an element $g$ in $E$, we denote $\tilde{g}$ any lift of $g$ to $G_0$. Denote the image of the centraliser $Z(\tilde{g})\subset G_0$ in $E$ by $S_g$. It has the following description:
\[
S_g = \left\langle h \in E \mid \omega (g,h) =0 \right\rangle ,
\]
in other words, it consists of all elements $h$ in $E$ whose preimages in $G_0$ commute with $\tilde{g}$. In order to show that $\pi^* (\zeta)$ is unramified in the stable cohomology of $G_0$, it is thus sufficient, by Proposition \ref{pCriterionUnramified}, to show:
\begin{quote}
($\ast$) For any $\tilde{g} \in G_0$, the restriction of $\pi^* (\zeta )$ to $H^*_{\mathrm{s}}(Z(\tilde{g}), \Z /2)$ is induced from $H^*_{\mathrm{s}} (S_g /\langle g \rangle , \Z /2)$ via the natural homomorphisms
\[
Z(\tilde{g}) \to Z(\tilde{g})/\langle \tilde{g} \rangle \to S_g/\langle g \rangle .
\]
\end{quote}
If $g=0$, this is obvious since $\zeta$ comes from $H^*_{\mathrm{s}} (E, \Z /2 )$. Hence we will assume in the sequel that $g\neq 0$. Without loss of generality, we can also assume that $\dim \, S_g \ge 4$ since otherwise $\zeta$ restricts to zero on $S_g$, hence will also be zero on $Z(\tilde{g})$ whence $(\ast)$ is trivially verified. By \cite[Lemma 1.4, Prop. 1.8]{EKM08}, on the subspace $S_g\subset E$ the form $\omega$ will decompose as 
\[
\omega \mid_{S_g} = 0\mid_{\mathfrak{r}_g} \perp (\omega_{g})\mid_{\mathfrak{a}_g}
\]
where $\mathfrak{r}_g\subset S_g$ is the radical of $\omega\mid_{S_g}$ and $\omega_{g}$ is nondegenerate on the complement $\mathfrak{a}_g$. Now we have: 
\begin{enumerate}
\item[(i)]
The subspace $\mathfrak{a}_g$ of $S_g$ on which $\omega$ is nondegenerate is nontrivial because maximal isotropic subspaces of $E$ have dimension $3$ and $\dim \, S_g \ge 4$. 
\item[(ii)]
$1\le \dim \mathfrak{r}_g \le 2$: indeed, $g$ is contained in $\mathfrak{r}_g$, whence $\dim \mathfrak{r}_g\ge 1$; and since $\mathfrak{a}_g$ is nontrivial, $\dim \mathfrak{r}_g\ge 3$ would lead to the contradiction that there would be totally isotropic subspaces of $E$ of dimension strictly greater than $3$ (one could produce one such by taking the subspace generated by $\mathfrak{r}_g$ and any nonzero element in $\mathfrak{a}_g$).
\end{enumerate}

Now let $p_1, \dots , p_r$, $1\le r \le 2$, be coordinates on $\mathfrak{r}_g$ (hence these are elements in $\mathfrak{r}_g^{\vee}$), and let $q_1, \dots , q_a$ be coordinates on $\mathfrak{a}_g$. Moreover, as $\omega_{g}$ is nondegenerate on $\mathfrak{a}_g$, $g\neq 0$ and $\mathfrak{a}_g\neq 0$, $\dim \, \mathfrak{a}_g =2$ or $4$ ($6$ being excluded because $\omega$ is nondegenerate and $g\neq 0$). Writing the restriction of $\zeta$ to $S_g$ in these coordinates, we obtain 
\[
\zeta\mid_{S_g} = \zeta_{q_1, \dots , q_a}^{(4)} + \sum_I \zeta_{I; p_1, \dots , p_r}^{(1)}\zeta^{(3)}_{I; q_1, \dots , q_a} + \sum_J \zeta_{J; p_1, \dots , p_r}^{(2)}\zeta^{(2)}_{J; q_1, \dots , q_a} 
\]
where the notation means the following: the upper index in brackets indicates the degree of the form, and the lower index coordinates indicate the coordinates involved in that form; e.g. $\zeta^{(3)}_{I; q_1, \dots , q_a}$ is a $3$-form in the coordinates $q_1, \dots , q_a$ only, and so on. Note that summands of the form $\zeta_{J; p_1, \dots , p_r}^{(2)}\zeta^{(2)}_{J; q_1, \dots , q_a}$ can only occur if $\rk \, \mathfrak{r}_g=2$ and hence $\rk\, \mathfrak{a}_g=2$ as well because it cannot be rank $4$ as $g\neq 0$. Hence Lemma \ref{lMinimal} always applies to yield that each 
\[
\zeta_{q_1, \dots , q_a}^{(4)}, \quad \zeta^{(3)}_{I; q_1, \dots , q_a} , \quad \zeta^{(2)}_{J; q_1, \dots , q_a} 
\]
maps to zero in $H^*_{\mathrm{s}}(Z(\tilde{g}), \Z /2)$: indeed, the form $\omega\mid_{S_g}$ is the commutator form of the extension $Z(\tilde{g}) \to S_g$, and classes in $H^*_{\mathrm{s}}(S_g, \Z/2)= \Lambda^* S_g^*$ that can be written as $(\omega\mid_{S_g})\wedge$(something) are mapped to zero in $H^*_{\mathrm{s}} (Z(\tilde{g}), \Z/2)$. 

Thus ($\ast$) holds universally, and Step 2 is proved.

\medskip

Finally, for Step 3, note that there are generically free $E$- and $G_0$-representations $V_E$ and $V_{G_0}$ such that, denoting by a superscript $L$ the loci where the actions are free, $(V_{G_0}^L)/G_0$ maps dominantly to $(V_E^L)/E$, and the induced field extension $k(V_{G_0})^G \supset k(V_E)^E$ factors as
\[
\xymatrix{
k(V_{G_0})^G \ar[r]^{\simeq} &k(\mathcal{S})(t)\\
& k(\mathcal{S})\ar@{^{(}->}[u]\\
& k(V_E)^E \ar@{^{(}->}[u]
}
\]
where $\mathcal{S}$ is a Severi-Brauer scheme over $k(V_E)^E$ and $t$ is an indeterminate: indeed, one can take for $V_E$ any generically free $E$-representation, which is at the same time a $G_0$-representation via the homomorphism $G_0\to E$, and for $V_{G_0}$ one takes $W\oplus V_E$, where $W$ is a generically free $G_0$-representation in which the center $\Z/2$ of $G_0$ acts nontrivially via multiplication by scalars. Then the sought-for Severi-Brauer scheme $\mathcal{S}$ is $(\P(W)\oplus V_E^L)/G_0 \to (V_E^L)/E$, over which $(W\backslash \{0\} \oplus V_E^L)/G_0$ is a $k^*$-principal bundle (Zariski locally trivial). Hence the tower of fields above. By Amitsur's theorem \cite[4.5.1]{GS06}, the kernel of $\mathrm{Br}(k(V_E)^E) \to \mathrm{Br}(k(\mathcal{S}))$ is cyclic generated by the class of $\mathcal{S}$, and $\mathrm{Br}(k(\mathcal{S}))\to \mathrm{Br}(k(\mathcal{S})(t))$ is injective. Since the two-torsion in the Brauer group of the fields $\Lambda$ involved here is precisely $H^2 (\Lambda, \Z/2)$, the definition of stable cohomology given in formula (\ref{fUnstable1}) shows that the kernel of $\pi^* \colon H^2_{\mathrm{s}}(E, \Z/2) \to H^2_{\mathrm{s}}(G_0, \Z/2)$ is cyclic. Since it contains the nontrivial class $\omega$, it is generated by $\omega$.
\end{proof}

\begin{remark}\label{fOtherPrimes}
The phenomenon, on which the proof of Theorem \ref{tMain2} is based to a large extent, that $X_-$ as in Proposition \ref{pDecompositionUnderSL2} can fail to be surjective on some $\Lambda^r V$ with $r\ge m-1$ in cases where $p \le m$, is not related to $p=2$, but reoccurs for other primes: it is only to do with the fact that $(X_-)^p=0$ in characteristic $p$. We therefore strongly suspect that examples of the type given in Theorem \ref{tMain2} where $K_2^G$ fails to generate the ideal $K^G$, for some extraspecial group $G$, exist for all primes $p$.
\end{remark}

\end{document}